\documentclass[12pt]{amsart}

\newtheorem{theorem}{Theorem}[section]

\newtheorem{proposition}[theorem]{Proposition}

\newtheorem{corollary}[theorem]{Corollary}

\newtheorem{remark}[theorem]{Remark}

\newtheorem{definition}[theorem]{Definition}

\newcommand{\lra}{\longrightarrow}

\usepackage{amssymb}

\numberwithin{equation}{section}

\begin{document}

\title{$(t,\ell)$-stability and  coherent systems}


\date{\today}
\thanks{ }

\author{L. Brambila-Paz}

\address{CIMAT, Mineral de Valenciana S/N, Apdo. Postal 402, C.P. 36240. Guanajuato, Gto,
M\'exico}
\email{lebp@cimat.mx}

\author{O. Mata-Gutierrez}

\address{Departamento de Matem\'aticas, CUCEI, Universidad de Guadalajara. Av. Revoluci\'on 1500. C.P. 44430, Guadalajara, Jalisco,
M\'exico} \email{osbaldo.mata@academico.udg.mx, osbaldo@cimat.mx}

\begin{abstract} Let $X$ be a non-singular irreducible complex projective curve of genus $g\geq 2$. The concept of stability of coherent systems over $X$ depends on a positive real parameter $\alpha $, given then a (finite) family of moduli spaces of coherent systems.
We use  $(t,\ell)$-stability to prove the existence of coherent systems over $X$ that are $\alpha$-stable for all allowed $\alpha >0$.
\end{abstract}
\maketitle

\section{Introduction}

Let $X$ be a non-singular irreducible complex projective curve of genus $g\geq 2$.
A coherent system of type $(n,d,k)$ on $X$ is a pair $(E,V)$ where $E$ is a vector bundle on $X$ of
rank $n$ and degree $d$ and $V\subset H^0(X,E)$ is a linear subspace of dimension $k$.
For any real number $\alpha $ there is a concept of $\alpha $-stability and
there exist  moduli spaces $G(\alpha;n,d,k)$ of $\alpha$-stable coherent systems of type $(n,d,k)$ (see \cite{kn} and \cite{rv}).
A necessary condition for the non-emptiness of $G(\alpha;n,d,k)$ is
that $\alpha>0$. Thus, there is a family of moduli spaces $G(\alpha;n,d,k)$ of $\alpha$-stable coherent systems of type $(n,d,k)$ (see \cite{kn} and \cite{rv}) parameterised by $\mathbb{R}^+$. Moreover, there are finitely many critical values
$0=\alpha_0<\alpha_1<\cdots<\alpha_L$ of $\alpha$; as $\alpha$
varies, the concept of $\alpha$-stability remains constant between
two consecutive critical values.  We denote by
$G_0(n, d, k)$ (resp. $G_L(n, d, k)$) the moduli spaces corresponding to $0 < \alpha < \alpha _{1}$ (resp.
$\alpha > \alpha _L$). The moduli space $G_0(n, d, k)$ is related to the Brill-Noether loci, i.e. the subspaces of the moduli space of stable bundles consisting of those bundles with a prescribed number of sections (see \S 2). The study of coherent systems has been applied to prove, in some cases, the non-emptiness, irreducibility and the dimension of the Brill-Noether loci (see e.g. \cite{bg}).

Precise conditions for non-emptiness of $G(\alpha; n, d, k)$ are known when $k \leq n$ (see \cite[Theorem 3.3]{bgmmn}).
For general curves, the first author in \cite{mio} gives a necessary and sufficient condition for $ G (\alpha; n, d, n + 1) $ to be non-empty, and describes geometric properties of $G(\alpha; n, d, n+1)$  (see also \cite{bbn1}, \cite{but} and \cite{bbn}).
For $k > n+1$, much less is known. There are general results due to M. Teixidor i Bigas \cite{mon1}, and E. Ballico \cite{bal}; they give numerical conditions
that are sufficient for the non-emptiness of $G(\alpha; n, d, k)$. Teixidor i Bigas conditions are for generic curves and Ballico conditions are for very large degree.

One of the main tools used in \cite{bgmmn} and \cite{mio}, i.e. when  $k\leq n+1$,  was the existence of coherent systems $(E,V)\in G_L(n, d, k)$ that are $\alpha $-stable for all $\alpha >0$ allowed;  in particular,  the interest was on the non-emptiness of the scheme
$$
U(n,d,k):=\{(E,V)\in G_0(n,d,k)\,|\,(E,V) \mbox{ is $\alpha$-stable for all }\alpha>0 \mbox{ and  $E$ is stable }\}.
$$
The significance of $U(n,d,k)$ is further strengthened  by the fact that a necessary condition for Butler's conjecture (see \cite{but}) to hold, is the existence a generated coherent system  in $U(n,d,k)$
(see in \cite{loap}). It is possible that is also a sufficient condition but we will not develop this point here.

In this paper we introduce a new technique in the study of the non-emptiness of $U(n,d,k)$ when $k\geq n+2$ that allows to ensures the existence of coherent systems in $G_0(n,d,k)$ that are $\alpha$-stables for all $\alpha >0$. The technique make use of the concept of $(t,\ell)$-stability (see Definition \ref{def}), introduced by  M.S. Narasimhan and S. Ramanan in \cite{nr} (see also \cite{nr2}). The aim of this paper is to relate $(t,\ell)$-stability of the vector bundle $E$ with $\alpha $-stability of the coherent system $(E,V)$.

Write
$$\varepsilon=\begin{cases}
1 &\mbox{if} \  d\equiv g-1\bmod n\\
0&\mbox{otherwise},\end{cases}$$
For any positive integers $0\leq a\leq g-1-\varepsilon $  denote by $ A_{a}(n,d,k)$  the subscheme

$$ A_{a}(n,d,k):=\{(E,V)\in G_0(n,d,k): E \ \ \mbox{is}  \ (0,a)- \mbox{stable} \},$$

The next theorems (see Theorem \ref{teo1} and \ref{teo12}) provides a criterion for the non-emptiness of $U(n,d,k)$.
Let $M(n,d)$ be the moduli space of stable vector bundles over $X$ of degree $d$ and rank $n$.

\begin{theorem}\label{teo1}  Assume that $0\leq a\leq g-1-\varepsilon $ and $ A_{a}(n,d,k)\ne \emptyset$ where $d\geq 2ng+s$  and  $k\geq d+n(1-g)-t$ with $s,t, a$  integers such that $0\leq t\leq a$ and $2t-s\leq a$.
 Then $ A_{a}(n,d,k)\subset U(n,d,k)$ and $U(n,d,k) \ne \emptyset $. Moreover, if $k \leq d+n(1-g)$ then
 $ \emptyset \ne A_{a}(n,d,k)\subset U(n,d,k)$ and $U(n,d,k)$  has a component of the expected dimension and birational to a Grassmannian bundle over an open set of $M(n,d).$
\end{theorem}

Clifford's Theorem for $\alpha$-semistable coherent systems (see \cite{lnclif}) states that if $d\leq 2gn$, then $k\leq \frac{d}{2}+n$. Given $(n,d,k)$ denote by $\lambda $ the difference  $\lambda:= d-2(k-n)$.

\begin{theorem}\label{teo12} Assume that $0\leq a\leq g-1-\varepsilon $.  If  $d\leq 2gn$ and $\lambda  \leq a$ then  $A_{a}(n,d,k)\subset U(n,d,k).$  Moreover, if  $A_{a}(n,d,k)\ne \emptyset$ then $U(n,d,k)\ne \emptyset.$
\end{theorem}

For lower degrees the non-emptiness of $A_{a}(n,d,k)$ depends on the non emptiness of a Brill Noether locus, which, for many cases, is still an open problem.
Nevertheless, for rank $2$ and $3$ we prove (see Theorem \ref{bn2} and \ref{bn3})

\begin{theorem}\label{bn2} Assume $k=2+r$ with $r\geq 1$. If there exists an integer $0\leq a\leq g-1-\varepsilon$ such that
\begin{equation}\label{eqdelta}\max\left\{d-2g-a,\frac{d-a}2\right\}\le r<d-2g+\frac{g-a+\delta -3}{2+r},
\end{equation}
then $ A_{a}(2,d,k)\subset U(2,d,k)$. Moreover, $\emptyset \ne A_a(2,d,k)\subset U(2,d,k)$.
\end{theorem}

With the notation
$$\vartheta=\begin{cases}
1 &\mbox{if} \ d-a\equiv 0\mod 3, \\
-1 &\mbox{if} \  d-a\equiv 1\mod 3 \\
0&\mbox{otherwise},\end{cases}$$
we have the following theorem for rank $3$.

\begin{theorem}\label{bn3} Assume $k=3+r$  with $r\geq 1$.  If there exists an integer $0\leq a\leq g-1-\varepsilon$ such that
\begin{equation}\label{eqdelta2}\max\left\{d-3g-a,\frac{d-a}2\right\}\le r<d-3g+\frac{2g-2a-1-\vartheta}{3+r},
\end{equation}
then $ A_{a}(3,d,k)\subset U(3,d,k)$. Moreover, $\emptyset \ne A_a(3,d,k)\subset U(3,d,k)$.
\end{theorem}

 Our numerical conditions are for any curve and for coherent systems with a general or special bundles. Also they  include large and lowers degrees, so they cover part of those conditions in \cite{mon1} and \cite{bal}, but more importantly they extends beyond theirs conditions. Our methods give results for special curves, in particular for hyperelliptic curves, and also for coherent systems with a general or a special bundles  with values outside Teixidor's parallelograms (see $[10,\S 5]$ and Remark \ref{rembn1}). Furthermore, since we
do not use the results of \cite{mon1} and \cite{bal}, our results give another proof of non-emptiness for those parts cover by  Teixidor i Bigas and E. Ballico  which are included in our results.

In Section $2$ we give some of the relevant results of the theories of Brill-Noether and coherent systems.
In Section $3$ we recall the main results on  $(t,\ell)$-stability that we will use; and we then prove our main results.

{\small Acknowledgments: The authors are member of the research group VBAC and  gratefully acknowledge the many helpful suggestions of P.E. Newstead and Angela Ortega
during previous versions of the paper.  The authors wishes to express their thanks to the referee(s) for several helpful comments and the suggested changes, all of which have substantially improved this version. The first author acknowledges  the support of CONACYT Proj $0251938$  and the second author thanks the partial support of PROSNI programe of Universidad de Guadalajara.}

\section{Brill-Noether theory and Coherent systems}

In this section we recall the main results that we will use on the Brill-Noether Theory and on coherent systems over a non-singular irreducible complex projective curve $X$ of genus $g\geq 2$.  For a more complete treatment of the subjects, see \cite{bgmn} and \cite{news} and \cite{ivmon} and the bibliographies therein.

\subsection{Brill-Noether Theory}

Let $M(n,d)$ (resp.\ $\widetilde{M}(n,d)$) denote the moduli space of stable (resp. S-equivalence classes of semistable) bundles of rank $n$ and degree $d$ on $X$. The {\it Brill-Noether loci} are defined by
$$B(n,d,k):=\{E\in M(n,d)\,|\, h^0(E)\ge k\},$$
$$\widetilde{B}(n,d,k):=\{[E]\in\widetilde{M}(n,d)\,|\,h^0(gr E)\ge k\},$$
where $[E]$ denotes the S-equivalence class of $E$ and $gr E$ is the graded object associated with $E$ through a Jordan-H\"older filtration. Since the Brill-Noether loci $B(n,d,k)$ are defined as determinantal varieties they are locally closed subschemes of expected dimension $$\rho(n,d,k):=n^2(g-1)+1-k(k-d+n(g-1)).$$ The number $\rho(n,d,k)$ is often referred to as the Brill-Noether number for $(g,n,d,k).$
We see at once that:
\begin{enumerate}
\item  if $n(g-1)<d$ and $0\leq k\leq d+n(1-g)$ then $B(n,d,k)=M(n,d)$.
    \item If $n(g-1)<d<2n(g-1)$ and $k>d+n(1-g)$, $B(n,d,k)\varsubsetneqq M(n,d)$.
    \item If $0<d\leq n(g-1)$ for any $k\geq 1$,  $B(n,d,k)\varsubsetneqq M(n,d).$
\end{enumerate}

Recall that a semistable vector bundle $E\in \widetilde{M}(n,d)$  is called special if $h^0(E)\cdot h^1(E)\ne 0$.

\begin{remark}\begin{em}\label{rembn1}
\begin{enumerate}
\item The special bundles are also called Brill-Noether bundles.
\item The problem of the non emptiness of a Brill-Noether locus, for many cases, is still an open problem. In \cite{bmno} it is represented in the Brill-Noether map the values of $(n,d,k)$ for which $B(n,d,k)\ne M(n,d)$ is not empty (see also \cite{ivmon}).
    \item The numerical conditions  in \cite{mon}, which are the same as those in \cite{mon1}, define the so-called Teixidor's parallelograms in the Brill-Noether map. In particular, in \cite[\S 5]{bmno} one can see the existence of values $(n,d,k)$ outside the Teixidor's parallelograms with $B(n,d,k)\ne \emptyset $ (see e.g.  \cite[Figure 6]{bmno}  for genus $g=10$).
    \end{enumerate}
\end{em}\end{remark}

Clifford's Theorem for special bundles (see \cite{bgn}) gives the bound $h^0(E)\leq \frac{d}{2} +n$.
For a special bundle $E\in \widetilde{M}(n,d)$ with $d\geq n(g-1)$ it follows immediately that:
\begin{enumerate}
\item $E^*\otimes K$ is special of degree $\leq n(g-1)$;
\item $h^1(E)\leq ng -\frac{d}{2};$
\item $h^0(E)=k_0+\imath $ for some  $\imath =1,\dots ,ng -\frac{d}{2}$ and $k_0=d+n(1-g)$.
\end{enumerate}

The Brill-Noether loci define a natural filtration
$$\dots{B}(n,d,k)\subseteqq  {B}(n,d,k-1)\subseteqq\dots \subseteqq{B}(n,d,1)\subseteqq{B}(n,d,0)=M(n,d). $$
$$\dots\widetilde{B}(n,d,k)\subseteqq\widetilde{B}(n,d,k-1)\subseteqq\dots \subseteqq\widetilde{B}(n,d,1)\subseteqq\widetilde{B}(n,d,0)=\widetilde{M}(n,d), $$
called the Brill-Noether filtration or just the BN-filtration on $M(n,d)$ (resp. in $\widetilde{M}(n,d)$). Note that if $B(n,d,k)\varsubsetneqq M(n,d)$,
${B}(n,d,k+1)\subset \mbox{Sing}{B}(n,d,k)$, and for many cases (see  \cite{ivmon}) ${B}(n,d,k+1)= \mbox{Sing}{B}(n,d,k)$ and $B(n,d,k)$ has a component of the expected dimension.

Denote by
$Y^{n,d}_k$, or simply by $Y_k$ when $(n,d)$ are understood, the scheme given by
$$Y^{n,d}_k:=B(n,d,k)-B(n,d,k+1).$$
Note that for any $E\in Y_k$, $h^0(E)=k.$

Such schemes $\{Y_k\}$ define a schematic stratification (see \cite{barba} or \cite{acgh})  on $M(n,d)$.
Let $\pi _2:X\times M(n,d)\rightarrow M(n,d)$ be the projection in the second factor.
Working locally in the \'etale topology if necessary, we can assume without loss of
generality that there exists a universal family $\mathcal{U}$ over $X\times M(n,d)$.
Let $\mathcal{U}_k$ be the restriction of $\mathcal{U}$ to $X\times Y_k$.
 The sheaf $\pi_{2*}(\mathcal{U}_k)$ is locally free of rank
$k$.  Moreover, the Grassmannian bundle $Grass(s,\pi_{2*}{\mathcal{U}_k})$  of $s$-dimensional subspaces has dimension $$\dim Grass(s,\pi_{2*}{\mathcal{U}_k})= \dim Y_k +s(k-s).$$

\begin{remark}\begin{em}\label{rembn} Let $k_0:=d+n(1-g)$.
\begin{enumerate}
\item If $d> 2n(g-1)$ then  $\pi _{2*}\mathcal{U}$ is locally free sheaf of rank $d+n(1-g).$ Moreover, $$\dim Grass(k,\pi_{2*}{\mathcal{U}})= \rho (n,d,k).$$
\item If $d\geq n(g-1)$ then $\emptyset \ne Y_{k_0}$ is an open set and for $k\leq k_0$, $$\dim Grass(k,\pi_{2*}{\mathcal{U}_{k{_0}}})= \rho(n,d,k).$$
\item For $k_0+\imath$ with $\imath =1,\dots ,ng -\frac{d}{2}$,  $$\dim Grass(k,\pi_{2*}{\mathcal{U}_{k{_0}+\imath}})= \dim {Y_{k{_0}+\imath}}+k(k_0-k)+k\imath.$$
\end{enumerate}
\end{em}\end{remark}

\subsection{Coherent systems}

Let $(E,V)$ be a coherent system of type
$(n,d,k)$ on $X$.  A subsystem of $(E,V)$ is a coherent system $(F,W)$ such that
$F\subset E$ is a subbundle of $E$ and $W\subset H^0(F)\cap V$. For a real number $\alpha >0$,  the $\alpha $-slope of a coherent system $(E,V)$ of type $(n,d,k)$, denoted by $\mu _{\alpha}(E,V)$, is the quotient
$$\mu _{\alpha}(E,V):= \frac{d+\alpha k}{n}.$$
A coherent system $(E,V)$ is  $\alpha$-stable (resp.  $\alpha$-semistable) if, for all proper subsystems $(F,W)$,
$$\mu _{\alpha}(F,W) <\mu _{\alpha}(E,V) \ \  \ \ (\mbox{resp.} \   \leq ). $$
We denote by $G_0(n,d,k)$  the moduli spaces of $\alpha$-stable coherent systems corresponding to
small $\alpha >0$ and by  $U(n,d,k)$  the subscheme
$$
U(n,d,k):=\{(E,V)\in G_0(n,d,k)\,|\,(E,V) \mbox{ is $\alpha$-stable for all }\alpha>0 \mbox{ and  $E$ is stable }\}.
$$

The Clifford's Theorem for $\alpha$-semistable coherent systems (see \cite{lnclif}) states that, for any $\alpha$-semistable coherent system $(E,V)$ of type $(n,d,k)$,
\begin{equation}\label{eqcliff1}
{k}\leq\begin{cases}
d +n(1 -g) &\mbox{if} \  d\geq 2gn\\
\frac{d}{2}+n&\mbox{if} \  d<2gn.\end{cases}
\end{equation}

There is a forgetful morphism
$$\Phi:G_0(n,d,k)\lra \widetilde{{B}}(n,d,k):(E,V)\mapsto [E].$$

\begin{remark}\begin{em}\label{nou} An easy computation shows that:
\begin{enumerate}
\item if $E\in M(n,d)$ is stable, then,
for any linear subspace $V\subset H^0(E)$ of dimension $k$,  $(E,V)\in G_0(n,d,k)$.
\item If $E\in B(n,d,k)$,  $\Phi ^{-1}(E)=Grass(k,H^0(E))$.
\item  If $E\in B(n,d,k)$ then $(E,V)\in U(n,d,k)$ if for all subsystems of type $(n',d',k')$, $\frac{k'}{n'}\leq \frac{k}{n}$. Moreover,
if $(E,V)\in G_0(n,d,k)$ but $(E,V)\notin U(n,d,k)$, then there exists an $\alpha_i>0$ and an $\alpha_i$-semistable coherent subsystem $(F,W)$ of type $(n',d',k')$, such that $\frac{k}{n} \leq \frac{k'}{n'}$.
\end{enumerate}
\end{em}\end{remark}

It is well known that if $d\geq 2n(g-1)$ and $k\leq d+n(1-g)$ then $G_0(n,d,k)$ is
birational to the Grassmannian bundle $Grass(k,\pi_{2*}{\mathcal{U}})$ and $\dim G_0(n,d,k) = \rho (n,d,k)$. Moreover, if $k_0=d+n(1-g)$, from Remark \ref{rembn},$(2)$, $$\dim \Phi ^{-1}(Y_{{k_0}}) = \dim Grass(k,\pi_{2*}\mathcal{U}_{{{k_0+\imath}}})=  \rho (n,d,k).$$
The following proposition computes the dimension of $\Phi ^{-1}(Y_{{k_0+\imath}})\subset G_0(n,d,k)$ for $d\geq n(g-1)$ and $\imath=0,1,\dots ,ng-\frac{d}{2}$.

\begin{proposition} Let $d\geq n(g-1)$ and $k_0=d+n(1-g)$. If $\imath=0,1,\dots ,ng-\frac{d}{2}$ and $c=\dim M(n,d)-\dim Y_{{k_0+\imath}}$ then for any $0\leq k\leq k_0+\imath$,
$$\dim \Phi ^{-1}(Y_{{k_0+\imath}}) = \rho (n,d,k) +k\imath -c.$$
 \end{proposition}

 \begin{proof} We know that $\Phi ^{-1}(Y_{{k_0+\imath}})\cong Grass(k,\pi_{2*}\mathcal{U}_{{{k_0+\imath}}})$, where $Grass(k,\pi_{2*}{\mathcal{U}}{_{k_0+\imath}})$ is a Grassmannian bundle of rank $k({{k_0+\imath}}-k)$ over $Y_{{k_0+\imath}}$.

If $c=\dim M(n,d)-\dim Y_{{k_0+\imath}}$ then
$$
\begin{array}{cll}
\dim \Phi ^{-1}(Y_{{k_0+\imath}})&= &\dim Grass(k,\pi_{2*}{\mathcal{U}}{_{k_0+\imath}})\\
&=& \dim Y_{{k_0+\imath}} +k({{k_0+\imath}}-k)\\
&=&\dim M(n,d)-c -k(k-d+n(g-1)) +k\imath\\
&=& \rho (n,d,k) +k\imath -c ,
\end{array}
$$
and this is precisely the assertion of the proposition.

\end{proof}

\section{$(t,\ell)$-stability and Main Results}

In this section we summarize without proofs the
relevant material on $(t,\ell)$-stability.  For a deeper discussion of $(t,\ell)$-stable bundles we refer the
reader to \cite{nr2} and \cite{osb} (see also \cite{nr}).

\begin{definition}\label{def}
Let $t,\ell \in \mathbb{Z}$. A vector bundle $E$ of rank $n$ and degree $d$ is $(t,\ell)$-stable  if, for all proper subbundles $F \subset E$,
$$
\frac{d_F +t}{n_F} < \frac{d+t-\ell}{n}.
$$
\end{definition}

Denote by $A_{t,\ell}(n,d)$ the set of $(t,\ell)$-stable bundles of rank $n$ and degree $d$. It is known that $(t,\ell)$-stability is an open condition \cite[Proposition 5.3]{nr2} and that $A_{t,\ell}(n,d)\ne \emptyset$ if and only if
\begin{equation}\label{eq1}
  t(n-r)+r\ell<r(n-r)(g-1)+\delta_r
\end{equation}
for all integers $r$ with $1\le r\le n-1$, where $\delta_r$ is the unique integer such that $0\le \delta_r\le n-1$ and $r(n-r)(g-1)+\delta_r\equiv rd\bmod n$ \cite[Proposition 1.9]{osb}.

We are interested in the relation between the $(0,a)$-stable bundles and $\alpha $-stable coherent systems. Write
$$\varepsilon=\begin{cases}
1 &\mbox{if} \  d\equiv g-1\bmod n\\
0&\mbox{otherwise},\end{cases}$$

\begin{proposition}
For any $ 0\leq a\leq g-1-\varepsilon$, $A_{0,a}(n,d)\ne \emptyset$ and it is an open set of the moduli space $M(n,d).$ Moreover  $A_{0,g-1}(n,d)\ne\emptyset$ if and only if $d\not\equiv g-1\bmod n$.
\end{proposition}

\begin{proof} From the inequalities $(\ref{eq1})$ we have that
for any $ 0\leq a\leq g-1-\varepsilon$, $A_{0,a}(n,d)\ne \emptyset$.  The $(0,a)$-stability implies that
\begin{equation}\label{eq2}
\mu (F) < \mu(E) -\frac{a}{n}  \qquad i.e. \ \ \frac{a}{n}<\mu(E) - \mu (F)
\end{equation}
for all subbundles of $E$. Therefore, $(0,a)$-stability implies stability.
\end{proof}

We have a filtration of open sets
$$\emptyset\ne A_{0,g-1-\varepsilon}(n,d)\subset \dots \subset A_{0,1}(n,d)\subset A_{0,0}(n,d)= M(n,d).$$

Denote by $A_{a}(n,d,k)$ the open subscheme
$$ A_{a}(n,d,k):=\{(E,V)\in G_0(n,d,k): E \ \ \mbox{is}  \ (0,a)- \mbox{stable} \}.$$
If $\Phi:G_0(n,d,k)\lra \widetilde{{B}}(n,d,k)$ is the forgetful map then
$$\Phi(A_{a}(n,d,k))= A_{(0,a)}(n,d)\bigcap B(n,d,k).$$   We see at once that $A_{a}(n,d,k)\ne \emptyset$ in the following cases.

\begin{proposition}\label{prop31} If $\dim A_{(0,a)}(n,d)^c  <  \dim {B}(n,d,k)$ then $A_{a}(n,d,k) \ne \emptyset .$
Moreover, if $d\geq n(g-1)$ and $k\leq d+n(1-g)$ then for any $ 0\leq a\leq g-1-\varepsilon$, $A_{a}(n,d,k) \ne \emptyset .$
\end{proposition}

\begin{proof}  We only need to make the following observation.
If $A_{(0,a)}(n,d)\bigcap B(n,d,k)\ne \emptyset $ then $A_{a}(n,d,k) \ne \emptyset .$ The hypotheses in the proposition give $A_{(0,a)}(n,d)\bigcap B(n,d,k)\ne \emptyset .$
\end{proof}

\begin{remark}\begin{em}\label{more} We have proved more, namely that if $\dim A_{(0,a)}(n,d)^c  <  \dim Y_r$ then $A_{a}(n,d,k)\ne \emptyset$ and for $r\geq k$,
$$Grass(k,\pi_{2*}{\mathcal{U}_r}){|_{{Y_{r}\bigcap  A_{(0,a)}(n,d)}}}\subset A_{a}(n,d,k)\subset  G_0(n,d,k).$$ Moreover, $\dim Y_r +k(r-k)\leq \dim  G_0(n,d,k)$.
\end{em}\end{remark}

The following theorems establish a relation between $(0,a)$-stable bundles and $\alpha $-stable coherent systems with $\alpha >0$. Moreover,   they ensures, under certain conditions, the existence of coherent systems in $G_0(n,d,k)$ that are $\alpha$-stables for all $\alpha >0$. 

 From now on, $a$ will be a positive integer such that $0\leq a<g-1-\varepsilon$.

\begin{theorem}\label{teo1}  Assume $ A_{a}(n,d,k)\ne \emptyset$ where $d\geq 2ng+s$  and  $k\geq d+n(1-g)-t$ with $s,t$  integers such that $0\leq t\leq a$ and $2t-s\leq a$.
 Then $ A_{a}(n,d,k)\subset U(n,d,k)$ and $U(n,d,k) \ne \emptyset $. Moreover, if $k \leq d+n(1-g)$ then
 $ \emptyset \ne A_{a}(n,d,k)\subset U(n,d,k)$ and $U(n,d,k)$  has a component of the expected dimension and birational to a Grassmannian bundle over an open set of $M(n,d).$
\end{theorem}

\begin{proof} Let $(E,V) \in A_{a}(n,d,k)$. We shall prove that under the hypothesis given $(E,V)$ is $\alpha$-stable for all $\alpha >0$.

Suppose for a contradiction that $(E,V)\notin U(n,d,k)$. From Remark \ref{nou}$(4)$
 there
exists an $\alpha_i$-semistable coherent subsystem $(F,W)$ of type $(n',d',k')$, such that $\frac{k}{n} \leq \frac{k'}{n'}$.

By hypothesis, one has
$$
\frac{d+n(1-g) - t} {n}  \leq  \frac{k}{n} \leq  \frac{k'}{n'}.
$$

 Assume $\mu(F) \geq 2g$. The Clifford bound \eqref{eqcliff1} for coherent systems  gives $\frac{k'}{n'}\leq \mu(F) +1 -g$. Using
this, together with the previous inequality, we obtain
$$
\mu(E)+1-g - \frac{t}{n} \leq  \frac{k'}{n'} \leq \mu(F)+1-g,
$$
which implies
$$
\mu(E)-\frac{a}{n} \leq \mu(E)-\frac{t}{n}\leq \mu(F),
$$
since $0\leq t\leq a$.
This contradicts the $(0,a)$-stability of $E$ (see \eqref{eq2}).

 Assume now $\mu(F)< 2g $. The Clifford bound for $(F,W)$ gives $\frac{k'}{n'} \le \frac{\mu(F)}{2} +1 $.
 Hence
 \begin{equation*}
 \mu(E)+1-g - \frac{t}{n}\leq \frac{k}{n} \leq  \frac{k'}{n'} \le \frac{\mu(F)}{2} +1.
\end{equation*}
So, since $E$ is $(0,a)$-stable,
$$\mu(E)-g - \frac{t}{n}\leq \frac{k}{n} \leq  \frac{\mu(F)}{2} <\frac{\mu(E)-\frac{a}{n}}{2} ,$$
which implies
$$
\mu(E)< 2g + 2\frac{t}{n}-\frac{a}{n}\leq 2g + \frac{s}{n}
$$
since by hypothesis $2t-s\leq a$.
 This contradicts the assumption that $d\geq 2ng+s .$ Hence,  $(E,V)\in U(n,d,k)$ as required.

 If $k\leq d+n(1-g)$, from Proposition \ref{prop31}, $A_{a}(n,d,k) \ne \emptyset $. Therefore the theorem follows from the observation that  $\Phi (A_{a}(n,d,k))$ is an open set of $M(n,d)$.

\end{proof}

\begin{remark}\begin{em}\label{special01}
\begin{enumerate}
\item Note that the theorem does not involve any assumptions about  $\Phi(G_0(n,d,k))$, it could be $M(n,d)$ or $\widetilde{{B}}(n,d,k)\ne M(n.d)$.
\item In Theorem \ref{teo1} the integer $s$ could be negative, and is bounded by $2t-a\leq s$. In this case, if $k > d+n(1-g)$ then, $\Phi(G_0(n,d,k))=\widetilde{{B}}(n,d,k)\ne M(n.d)$ and 
from Proposition \ref{prop31}, $A_{a}(n,d,k) \ne \emptyset $ if $\dim A_{(0,a)}(n,d)^c  <  \dim {B}(n,d,k).$
\item A slight change in the proof of Theorem \ref{teo1} actually shows that if $(E,V)\in A_a(n,d,k) $ with $E$ special and  $h^0(E)=d+n(1-g)+\imath$ then $(E,V)\in U(n,d,k)$ if $d\geq 2ng+2(t-\imath)-a$ and
 $ d+n(1-g)-t\leq k $ when $0\leq t-\imath\leq a.$
 \end{enumerate}
\end{em}\end{remark}

Clifford's Theorem for $\alpha$-semistable coherent systems of type $(n,d,k)$ and degree $0<d\leq 2gn$ implies that $k\leq \frac{d}{2}+n$. We denote by $\lambda $ the difference
$$\lambda := d-2(k-n).$$

\begin{theorem}\label{teo12}  If $0<d\leq 2gn$  and  $\lambda  \leq a$ then  $A_{a}(n,d,k)\subset U(n,d,k).$  Moreover, if  $A_{a}(n,d,k)\ne \emptyset$ then $U(n,d,k)\ne \emptyset.$
\end{theorem}

\begin{proof} Let $(E,V)\in  A_{a}(n,d,k).$
 Analysis similar to that in the proof of Theorem \ref{teo1} shows
that if $(E,V)\notin U(n,d,k)$ we get a contradiction. Indeed, suppose that
 there
exists an $\alpha_i$-semistable coherent subsystem $(F,W)\subset (E,V)$ of type $(n',d',k')$, such that $\frac{k}{n} \leq \frac{k'}{n'}.$
Since $E$ is $(0,a)$-stable, and hence stable,  $$\mu(F)< \mu (E) \leq 2g.$$ Thus, from Clifford's Theorem for coherent systems we have that $\frac{k'}{n'} \leq \frac{\mu(F)}{2} +1.$
Hence,
$$
\frac{\mu(E)}{2} -\frac{\lambda}{2n}+1 =  \frac{k}{n} \leq \frac{k'}{n'} \leq \frac{\mu(F)}{2} +1.
$$
The assumption $\lambda \leq a$ implies that

$$
\mu(E) \leq  \mu(F) +\frac{\lambda}{n}\leq \mu(F) +\frac{a}{n}
$$
which contradicts the $(0,a)$-stability of $E$. This gives $U(n,d,k)\ne \emptyset$, and the theorem follows.

\end{proof}

For rank $2$ and $3$, we can prove that $U(n,d,k)\ne\emptyset$ for a wider range of values of $d$ and $k$ by computing the dimension of $A_{0,a}(n,d)^c:=M(n,d)\setminus A_{0,a}(n,d)$.
An estimate for this was given in \cite[Theorem 1.10]{osb},
 but it is possible to compute it precisely using the Segre invariants. Recall  (see \cite{BL}) that the $m$-Segre invariant $s_m(E)$ of a bundle of rank $n$ and degree $d$ is defined by
$$s_m(E):=\min_{F\subset E}\{md-nd_F\ |\ F \mbox{ a  subbundle of rank $m$ of } E \},$$
that is,
\begin{equation}\label{segre1}
\frac{s_m(E)}{mn}=min_{F\subset E}\{\mu (E)-\mu(F)\}.
\end{equation}

Let $M(n,d,m,s)$ be the set of stable vector bundles of rank $n$ and degree $d$
such that the $m$-Segre invariant is $s$, that is
$$
M(n,d,m,s):=\{E\in M(n,d)\ |\ s_{m}(E)=s\}.
$$

In  \cite{RT} (see also  \cite{BL})  it was proved that for an integer  $0<s\leq m(n-m)(g-1)$
such that $s\equiv md\ mod\, n,$ $M(n,d,m,s)$ is non empty and irreducible and
$$
\dim \, M(n,d,m,s)=n^{2}(g-1)+1+s-m(n-m)(g-1).
$$

In the following result we describe the $A_{0,a}(n,d)$ in terms of Segre invariants. First, we introduce the following notation
\begin{eqnarray}\label{sm}
  \tilde{s}_{m}:=\max\{s\ | \ s\leq ma, \ \ s\equiv\, md\, mod\, n\},
\end{eqnarray}
and
\begin{eqnarray}\label{sd}
  s_{\Delta}:=\min_{m}\{m(n-m)(g-1)-\tilde{s}_{m}\}.
\end{eqnarray}

\begin{theorem}\label{teo3} For any $ 0\leq a\leq g-1-\varepsilon$,
  $$
  A_{0,a}(n,d)=\bigcap_{m=1}^{n-1}\left(\bigcup_{s>ma}M(n,d,m,s) \right).
  $$ Moreover, $dim\, A_{0,a}(n,d)=n^{2}(g-1)+1-s_{\Delta}$.
\end{theorem}
\begin{proof}
 The first part follows immediately from $(\ref{eq2})$ and $(\ref{segre1})$.

  The dimension of  $ (A_{0,a}(n,d))^{c}$  follows from the next equalities:
  \begin{eqnarray*}
    \dim (A_{0,a}(n,d))^{c}&=&\dim  \bigcap_{m=1}^{n-1}\left[\left(\bigcup_{s> ma}M(n,d,m,s) \right)\right]^{c}\\
    &=&\dim  \bigcup_{m=1}^{n-1}\left[\left(\bigcup_{s> ma}M(n,d,m,s) \right)^{c}\ \right]\\
    &=&\dim  \bigcup_{m=1}^{n-1}\left[\left(\bigcup_{s\leq ma}M(n,d,m,s) \right)\right]\\
    &=& \max_{m} \left\{ \max_{s}\{\dim\,(M(n,d,m,s))\}\right \}\\
    &=& \max_{m} \left\{ \max_{s}\{ n^{2}(g-1)+1+s-m(n-m)(g-1)\}\right \}\\
    &=& \max_{s}\{n^{2}(g-1)+ \tilde{s}_{m}-m(n-m)(g-1)\}\\
    &=& n^{2}(g-1)+1-s_{\Delta}.
  \end{eqnarray*}
\end{proof}

The following results are an application of Theorem \ref{teo3} for vector bundles of rank $2$ and $3$.

\begin{corollary}\label{la0a}
$\dim A_{0,a}(2,d)^c=3g+a-\delta,$
where
$$\delta=\begin{cases}
2&\mbox{if} \  a\equiv d \ \mbox{mod}\ 2\\
3&\mbox{otherwise,} \end{cases}$$
\end{corollary}
\begin{proof}
From Theorem \ref{teo3}
$$A_{0,a}(2,d)^c=\bigcup_{0\leq s\le a}M(2,d,s)$$
and
$$\dim M(2,d,s)=3g+s-2$$
for $s\le g-1$ (see also \cite[Proposition 3.1]{LN}). Since $s\equiv d\mod 2$, it follows that $\dim M(2,d,s)$ attains its maximum for $s\le a$ when $s=a$ if $a\equiv d\mod 2$ or when $s=a-1$ otherwise. The result follows.
\end{proof}

\begin{theorem}\label{bn2} Assume $k=2+r$ with $r\geq 1$. If there exists an integer $0\leq a\leq g-1-\varepsilon$ such that
\begin{equation}\label{eqdelta}\max\left\{d-2g-a,\frac{d-a}2\right\}\le r<d-2g+\frac{g-a+\delta -3}{2+r},
\end{equation}
then $ A_{a}(2,d,k)\subset U(2,d,k)$. Moreover, $\emptyset \ne A_a(2,d,k)\subset U(2,d,k)$.
\end{theorem}

\begin{proof} We begin by proving that $A_a(2,d,k)\ne \emptyset$. Since  $\dim B(2,d,k)\ge\beta(2,d,k)$, it is sufficient by Proposition \ref{prop31} to prove that $\dim A_{0,a}(2,d)^c  <  \beta(2,d,k).$
According to Corollary \ref{la0a}, this means we need to prove that
$$3g+a-\delta < 4(g-1)+1-(k)(r-d+2g).$$
This follows from the second inequality in \eqref{eqdelta}.

It remains to show that $A_a(2,d,k)\subset U(2,d,k)$.
For this, we argue as in the proof of Theorem \ref{teo1} and \ref{teo12}.  Let $(E,V)\in A_a(2,d,k)$ and suppose $(E,V)\not\in U(2,d,k)$. Let $(F,W)$
be a subsystem  of $(E,V)$ of type $(1,d',k')$ such that $\frac{k}{2}\leq  k'$. From\eqref{eqdelta}, we have $k\ge d+2-2g-a$. If $d'\geq 2g$,
$$
\mu(E)+1-g - \frac{a}{n} \leq  {k'} \leq \mu(F)+1-g,
$$
which implies
$$
\mu(E) \leq  \mu(F)+\frac{a}{n}.
$$
This contradicts the $(0,a)$-stability of $E$. If $d'<2g$,
 $$\frac{k}{2}\leq \frac{d'}2+1\leq \frac{d-a}4+1.$$
This contradicts the first inequality in \eqref{eqdelta}. Hence, $\emptyset \ne A_a(2,d,k)\subset U(2,d,k)$ as claimed.
\end{proof}

For rank $3$ Theorem \ref{teo3} gives three different cases.

\begin{corollary}\label{lem3}If $0\leq a\leq g-1-\varepsilon$ then $\dim A_{0,a}(3,d)^{c}=7(g-1)+1+\tilde{s}_{2}$,  with
  $\tilde{s_{2}}=\max\{s|s\leq 2a,\ \ s\equiv 2d\mod 3\}$. Moreover,
  \begin{enumerate}
    \item if $d-a\equiv 0\mod 3$  then $dim\, A_{0,a}(3,d)^{c}=7(g-1)+2a+1$;\\

    \item if $d-a\equiv 1\mod 3$ then  $dim\, A_{0,a}(3,d)^{c}=7(g-1)+2a-1$;  \\

    \item if   $d-a\equiv 2\mod 3$ then $dim\, A_{0,a}(3,d)^{c}=7(g-1)+2a$.  \\
  \end{enumerate}

\end{corollary}

\begin{proof}
 By hypothesis we have that $m=1,2$. Now, using (\ref{sm}) and (\ref{sd}) we have $\tilde{s_1}\leq a$
 with $s_{1}\equiv d\mod 3$ and $\tilde{s}_{2}\leq 2a$ with $\tilde{s_{2}}\equiv 2d\mod 3$. Therefore
$\tilde{s_{1}}\leq \tilde{s_{2}}$ and $s_{\Delta}= 2(g-1)-\tilde{s}_{2}.$ Now, the result follows
from Theorem  \ref{teo3}.

\end{proof}

With the notation
$$\vartheta=\begin{cases}
1 &\mbox{if} \ d-a\equiv 0\mod 3, \\
-1 &\mbox{if} \  d-a\equiv 1\mod 3 \\
0&\mbox{otherwise},\end{cases}$$
we have the following theorem for rank $3$.

\begin{theorem}\label{bn3} Assume $k=3+r$ with $r\geq 1$.  If there exists an integer $0\leq a\leq g-1-\varepsilon$ such that
\begin{equation}\label{eqdelta2}\max\left\{d-3g-a,\frac{d-a}2\right\}\le r<d-3g+\frac{2g-2a-1-\vartheta}{3+r},
\end{equation}
then $ A_{a}(3,d,k)\subset U(3,d,k)$. Moreover, $\emptyset \ne A_a(3,d,k)\subset U(3,d,k)$.
\end{theorem}

\begin{proof} As in Theorem \ref{bn2} we begin by proving that $A_a(3,d,k)\ne \emptyset .$ If we prove that $\dim A_{0,a}(3,d)^c  <  \dim B(3,d,k) $, the assertion follows.

It is easily seen that we can conclude from the second inequality in \eqref{eqdelta2} that $$7(g-1)+2a + \vartheta < 9(g-1)+1-k(r-d+3g),$$
hence that $\dim A_{0,a}(3,d)^c  <  \beta(3,d,k)\leq \dim B(3,d,k) $, and finally that $A_a(3,d,k)\ne \emptyset .$

To show that $A_a(3,d,k)\subset U(3,d,k)$ we argue as in the proof of Theorem \ref{teo1}, \ref{teo12} and \ref{teo3}. We leave it to the reader to verify that if $(E,V)\in A_a(3,d,k)$ and $(E,V)\not\in U(3,d,k)$ we get a contradiction using the first inequality in \ref{eqdelta2}.

\end{proof}

\end{document}